\documentclass{scrartcl}
\usepackage[utf8]{inputenc}
\usepackage[T1]{fontenc}
\usepackage{amsmath,amsthm,amssymb}
\usepackage{newunicodechar}
\usepackage{nicefrac}
\newunicodechar{ε}{\varepsilon}
\newunicodechar{Δ}{\Delta}
\newunicodechar{ℝ}{\mathbb{R}}
\newunicodechar{∂}{\partial}
\newunicodechar{→}{\to}
\newunicodechar{ν}{\nu}
\newunicodechar{∞}{\infty}
\newunicodechar{φ}{\varphi}
\newunicodechar{ψ}{\psi}
\newunicodechar{θ}{\theta}
\newunicodechar{Φ}{\Phi}
\newunicodechar{ϕ}{\phi}
\newunicodechar{ρ}{\rho}
\newunicodechar{σ}{\sigma}
\newunicodechar{π}{\pi}
\newunicodechar{τ}{\tau}
\newunicodechar{χ}{\chi}
\newunicodechar{Γ}{\Gamma}
\newcommand{\norm}[2][]{\|#2\|_{#1}}
\newcommand{\betrag}[1]{\left\lvert#1\right\rvert}
\newcommand{\ue}{u_{ε}}
\newcommand{\ve}{v_{ε}}
\newcommand{\we}{w_{ε}}
\newcommand{\De}[1]{D_{ε}[#1]}
\newcommand{\Det}[1]{\widetilde{D_{ε}}[#1]}
\newcommand{\Qe}{Q_{ε}}
\newcommand{\Ee}{E_{ε}}
\newcommand{\f}[2]{\frac{#1}{#2}}
\newcommand{\nf}[2]{\nicefrac{#1}{#2}}
\newcommand{\nn}{\nonumber}
\newcommand{\Om}{\Omega}
\newcommand{\Ombar}{\overline{\Om}}
\newcommand{\Lom}[1]{L^{#1}(\Om)}
\newcommand{\Ldom}[1]{L^{#1}(∂\Om)}
\newcommand{\kl}[1]{\left(#1\right)}
\newcommand{\kkl}[1]{\left[#1\right]}
\newcommand{\set}[1]{\left\{#1\right\}}
\newcommand{\bdry}{\big|_{∂\Om}}
\newcommand{\ballexterior}{ℝ^3\setminus B_1(0)}
\newcommand{\charact}[1]{χ_{#1}}
\newcommand{\ds}{\,\mathrm{d} s}
\newcommand{\dz}{\,\mathrm{d}z}
\newcommand{\dy}{\,\mathrm{d}y}
\newcommand{\drho}{\,\mathrm{d}ρ}
\newcommand{\dsigma}{\,\mathrm{d}σ}

\title{Rate of convergence in the large diffusion limit for the heat equation \\
with a dynamical boundary condition}
\author{
Marek Fila\\
\small Department of Applied Mathematics and Statistics,
\small Comenius University\\[-2mm]
\small 84248 Bratislava, Slovakia\\[-2mm]
\small e-mail: fila@fmph.uniba.sk
\\
\\
Kazuhiro Ishige\\
\small Graduate School of Mathematical Sciences,
\small The University of Tokyo\\[-2mm]
\small 3-8-1 Komaba, Meguro-ku, Tokyo 153-8914, Japan\\[-2mm]
\small e-mail: ishige@ms.u-tokyo.ac.jp
\\
\\
Tatsuki Kawakami\\
\small Department of Applied Mathematics and Informatics, 
\small Ryukoku University\\[-2mm]
\small Seta Otsu 520-2194, Japan\\[-2mm]
\small e-mail: kawakami@math.ryukoku.ac.jp
\\
\\
Johannes Lankeit\\
\small Department of Applied Mathematics and Statistics,
\small Comenius University\\[-2mm]
\small 84248 Bratislava, Slovakia\\[-2mm]
\small e-mail: jlankeit@math.upb.de
}
\date{}

\newtheorem{thm}{Theorem}
\newtheorem{lemma}[thm]{Lemma}
\newtheorem{defn}[thm]{Definition}
\newtheorem{remark}[thm]{Remark}
\begin{document}
\maketitle 
\begin{abstract}
\noindent 
We study the heat equation on a half-space or on an exterior domain with a linear dynamical boundary condition. 
Our main aim is to establish the rate of convergence to solutions of the Laplace equation with the same dynamical
boundary condition as the diffusion coefficient tends to infinity.\\
 \textbf{Keywords:} heat equation, dynamical boundary condition, large diffusion limit\\
 \textbf{MSC (2010):} 35K05, 35B40
\end{abstract}
\section{Introduction}\label{sec:intro}
We consider the heat equation with a dynamical boundary condition, 
\begin{align}\label{parabolic}
 ε∂_t\ue -Δ\ue&= 0 & &\text{in } \Om\times(0,∞),\\
 ∂_t\ue+∂_{ν}\ue &= 0 & & \text{on } ∂\Om\times(0,∞),\nn\\
 \ue(\cdot,0)&=φ & & \text{in } \Om,\nn\\
 \ue(\cdot,0)&=φ_b & & \text{on } ∂\Om\nn
\end{align}
and the 
connection with its elliptic counterpart,  
\begin{align}\label{elliptic}
 -Δu & = 0 & &\text{in } \Om\times(0,∞),\\
 ∂_tu+∂_{ν}u &= 0 & & \text{on } ∂\Om\times(0,∞),\nn\\
 u(\cdot,0)&=φ_b & & \text{on } ∂\Om.\nn
\end{align}
Here, $φ\in \Lom{∞}$, $φ_b\in\Ldom{∞}$, $ε\in(0,1)$.

The expected relation is, of course, that solutions of \eqref{parabolic} converge to those of \eqref{elliptic} as $ε→0$. 
Indeed, for the case of the half-space $\Om=ℝ^N_+:=ℝ^{N-1}\timesℝ_+$, $N\ge2$, this
has been proven recently in \cite{fila_ishige_kawakami_largediffusionlimit}.  This means, in particular, that the influence
of the initial function $φ$ is lost in the limit. In the present paper we are interested in a detailed description of the loss
of influence of $φ$. More precisely, we investigate the rate of convergence of solutions of 
\eqref{parabolic} to solutions of \eqref{elliptic} as $ε→0$.


We will deal with the case $\Om=ℝ^N_+$, $N\ge 2$, and the radially symmetric setting for $\Om=ℝ^3\setminus
B_1(0)$, $B_1(0):=\{x\in ℝ^3~|~|x|<1\}$. It turns out that the rate of convergence is of order $\sqrt{ε}$
in both cases.

For bounded domains $\Om$, results on convergence as $ε→0$ were established in \cite{GC3} by a method that is completely
different from the one used in \cite{fila_ishige_kawakami_largediffusionlimit} and in this paper. The rate of convergence was not
studied in \cite{GC3}.

Various aspects of analysis of parabolic
equations with dynamical boundary conditions have been treated by many authors,
for existence, uniqueness and regularity see for example
\cite{AQR, EMR, E2, E3, FV, H, VV1, V1}, for blow-up of solutions \cite{FV, KI, BP1, BBR, BP2, BPR}, for asymptotic behaviour of
solutions  \cite{GC3, GC2, GS, WH, BC, FQ2}, and for the mean curvature flow with dynamical boundary
conditions see \cite{GH, GOT}. Some of similar issues for elliptic equations with dynamical boundary conditions
were considered in
\cite{KI, E1, E4, FP, FIK01, FIK02, FIK03, FIK04, FIK05, FQ1, GM, KNP1, KNP2, L, V2, Y}, for instance.

Given $ϕ\in \Lom\infty$, by $S_1(t)ϕ:=θ(t)$ we denote the bounded 
 solution of
\begin{align*}
 ∂_t θ &= Δθ && \text{in } \Om\times(0,∞),\\
 θ &= 0 && \text{on } ∂\Om\times(0,∞),\\
 θ(\cdot,0)&= ϕ &&\text{in } \Om.
\end{align*}

For $ψ\in\Ldom\infty$, we want to have that $S_2(t)ψ:=p(t)$, where $p$ solves 
\begin{align*}
 -Δ p &= 0 && \text{in }\Om\times[0,∞),\\
 ∂_t p+∂_{ν} p&=0 && \text{on } ∂\Om\times(0,∞),\\
 p(\cdot,0)&=ψ&&\text{in } ∂\Om,
\end{align*}
and hence define 
\[
 \kkl{S_2(t)ψ}(x)= \f{2Γ(\f N2)}{\sqrt{π}Γ(\f{N-1}2)}\int_{ℝ^{N-1}} (x_N+t)^{1-N} \kl{1+\f{|x'-y'|^2}{(x_N+t)^2}}^{-\f N2} ψ(y') \dy',
\]
for $\Om=ℝ^N_+$, $x=(x',x_N)\in ℝ^{N-1}\times ℝ_+$, and 
\[
 \kkl{S_2(t)ψ}(r) = \f{ψ}{re^t}, \qquad r>1, t\ge 0.
\]
for radially symmetric $ψ\in \Ldom\infty$ (i.e. $ψ\in ℝ$) in case of $\Om=\ballexterior$.

We furthermore introduce 

\[
 F_1[φ_b](x,t) := ∂_t \kkl{S_2(t)φ_b} (x), \qquad φ_b\in \Ldom\infty,\quad x\in \Om, \; t>0,
\]
and 
\[
 F_2[v] (x,t) := - \kkl{S_2(0)∂_{ν}v\bdry(\cdot,t)}(x) - \int_0^t ∂_t \kkl{S_2(t-s)∂_{ν}v\bdry(\cdot,s)}(x) \ds 
\]
for $x\in \Om$, $t>0$ and functions $v\colon \Ombar\to ℝ$ that have a normal derivative $∂_{ν}v\in\Ldom\infty$ on $∂\Om$.

If 
\begin{equation}\label{defwe}
 \we(x,t) = \kkl{S_2(t)φ_b}(x) - \int_0^t \kkl{S_2(t-s)∂_{ν}\ve (s)}(x) \ds,
\end{equation}
then 
\[
 ∂_t \we = F_1[φ_b] + F_2[\ve]
\]
and 
\begin{equation}\label{weq}
 Δ\we=0, \quad \kkl{∂_t\we + ∂_{ν}\we}\bdry = \kkl{- ∂_{ν}\ve}\bdry, \quad \we(\cdot,0)\bdry=φ_b.
\end{equation}
If 
\begin{equation}\label{defve}
 \ve(x,t):=\kkl{S_1\kl{\nf t{ε}}Φ}(x) - \int_0^t \kkl{S_1\kl{\nf{(t-s)}{ε}}∂_t\we(s)}(x)\ds,
\end{equation}
where $Φ=φ-S_2(0)φ_b$, then $\ve$ solves 
\begin{equation}\label{veq}
 ε∂_t\ve = Δ\ve - ε∂_t\we = Δ\ve - εF_1[φ_b] -εF_2[\ve], \quad \ve\bdry = 0,\quad \ve(\cdot,0)=Φ,
\end{equation}
so that $\ue:=\ve+\we$ is a classical solution of \eqref{parabolic}. 

\begin{defn}\label{def:solution}
Let $φ\in \Lom{\infty}$ and $φ_b\in\Ldom{\infty}$. Let $T\in(0,∞]$ and 
\[
 \ve, ∂_{ν}\ve, \we \in C(\Ombar\times(0,T)).
\]
We call $(\ve,\we)$ a solution of \eqref{veq}, \eqref{weq} -- and $\ue:=\ve+\we$ a solution of \eqref{parabolic} -- in $\Om\times(0,T)$ if \eqref{defve} and \eqref{defwe} are satisfied. The solutions are called global-in-time solutions if $T=∞$. 
\end{defn}

Note: Here $ν$ is to be understood as (smooth) vector field defined on all of $\Ombar$, which on $∂\Om$ coincides with the outer unit normal. In the context of the cases treated in this article, of course, $∂_{ν} = -∂_r$ for $\Om=\ballexterior$ and $∂_{ν}=-∂_{x_N}$ for $\Om=ℝ^N_+$. 

The following was shown in \cite[Theorem~1.1 and Corrollary~1.1]{fila_ishige_kawakami_largediffusionlimit}: 
\begin{thm}\label{thm:halfspace}
 Let $N\ge2$ and $\Om=ℝ^N_+$; let $ε\in(0,1)$, $φ\in\Lom{∞}$ and $φ_b\in\Ldom{∞}$. Then problem \eqref{veq}, \eqref{weq} has a unique global-in-time solution $(\ve,\we)$ satisfying 
 \begin{equation}\label{halfspacenormXTfinite}
  \sup_{0<t<T} \kl{\norm[\Lom{∞}]{\ve(\cdot,t)}+\sqrt{\nf t{ε}}\norm[\Lom{\infty}]{∂_{ν}\ve(\cdot,t)}+\norm[\Lom{\infty}]{\we(\cdot,t)} } < \infty
 \end{equation}
 for any $T>0$. Furthermore, $\ve$ and $\we$ are bounded and smooth in $\Ombar\times I$ for any bounded interval $I\subset(0,∞)$.
 Moreover, for every $T>0$ there is $C(T)>0$ such that for every $ε\in(0,1)$
and $φ, φ_b$ as above we have
 \begin{align}\label{estimateXThalfspace}
  \sup_{0<t<T}& \kl{\norm[\Lom\infty]{\ve(\cdot,t)} + \sqrt{\nf{t}{ε}}\norm[\Lom\infty]{∂_{ν}\ve(\cdot,t)} + \norm[\Lom\infty]{\we(\cdot,t)} }\nn\\
  & \le C(T)\kl{\norm[\Lom\infty]{φ} + \norm[\Ldom\infty]{φ_b}}.
 \end{align}
 \\
 Finally, $\ue:=\ve+\we$ is a classical solution of \eqref{parabolic} and for every compact set $K$ in $\Ombar$ and $0<τ_1<τ_2<∞$ 
it holds that 
 \begin{equation}\label{thm:hspace-eqlimit}
  \lim_{ε\to0}\sup_{τ_1<t<τ_2} \norm[L^\infty(K)]{\ue(\cdot,t)-S_2(t)φ_b} = 0.
  \end{equation}
\end{thm}

In this paper we establish an analogous result for radially symmetric solutions when the domain $\Om$ is the exterior of the unit
ball in $ℝ^3$.
\begin{thm}\label{thm:ball}
Let $\Om=\ballexterior$, $ε\in(0,\nf{1}{\sqrt{π}})$ and let the functions $φ\in\Lom{∞}$ and $φ_b\in\Ldom{∞}$ be radially
symmetric.\\  
a) Then problem \eqref{veq}, \eqref{weq} has a unique radially symmetric global-in-time solution $(\ve,\we)$ satisfying \eqref{halfspacenormXTfinite} and $\ue:=\ve+\we$ is a classical solution of \eqref{parabolic}.\\
b) Furthermore, if 
 \begin{equation}\label{decayofphi}
\sup_{ρ\ge 1} |ρφ(ρ)|<\infty,
  \end{equation}
and $0<τ_1<τ_2<∞$, then 
 \begin{equation}\label{thm:ball-eqlimit}
  \lim_{ε\to0}\sup_{τ_1<t<τ_2} \norm[L^\infty(\Om)]{\ue(\cdot,t)-S_2(t)φ_b} = 0.
 \end{equation}
 Moreover, for every $ε_0\in(0,\nf1{\sqrt{π}})$ and $T>0$ there is $C(T)>0$ such that for every $ε\in(0,ε_0)$
 and $φ, φ_b$ as above
the following holds: 
\begin{align}\label{thm:ball-estimate}
    \sup_{0<t<T} &\kl{\norm[\Lom\infty]{\ve(\cdot,t)} + \sqrt{\nf{t}{ε}}\norm[\Lom\infty]{∂_{ν}\ve(\cdot,t)} + \norm[\Lom\infty]{\we(\cdot,t)} }\nn\\
    & \le C(T)\kl{\norm[\Lom\infty]{φ} + \norm[\Ldom\infty]{φ_b}+\sup_{ρ\ge 1} |ρφ(ρ)|}.
 \end{align}
\end{thm}

The role of assumption \eqref{decayofphi} is explained in Section~\ref{sec:remark-condition}.

Our remaining results are concerned with the question what the optimal rate of convergence in \eqref{thm:hspace-eqlimit}
and \eqref{thm:ball-eqlimit} is.

\begin{thm} \label{thm:upperbound-halfspace}
Let $\Om=ℝ^N_+$, $N\ge2$. Let $φ\in\Lom\infty$, $φ_b\in\Ldom\infty$, $K\subset\Ombar$ compact and $0<τ_1<τ_2<∞$. 
Then there is $C>0$ such that 
\[
 \sup_{τ_1<t<τ_2} \norm[L^\infty(K)]{\ue(\cdot,t)-S_2(t)φ_b}\le C\sqrt{ε},\qquad \text{for every }\; ε\in(0,1).
\]
\end{thm}

\begin{thm}\label{thm:upperbound-ball}
 Let $\Om=\ballexterior$. Let $φ\in\Lom\infty$ 
be radially symmetric and such that 
\eqref{decayofphi} holds. Assume further that $φ_b$ is constant and
$ε_0\in(0,\nf1{\sqrt{π}})$, $0<τ_1<τ_2<∞$. Then there is $C>0$ such that 
\[
 \sup_{τ_1<t<τ_2} \norm[\Lom\infty]{\ue(\cdot,t)-S_2(t)φ_b} \le C\sqrt{ε},\qquad \text{for every }\; ε\in(0,ε_0).
\]
\end{thm}

The upper bounds from Theorems~\ref{thm:upperbound-halfspace} and \ref{thm:upperbound-ball} are sharp.

\begin{thm}\label{thm:lowerbound-halfspace}
 Let $\Om=ℝ^N_+$, $N\ge 2$ and $φ_b\equiv 0$. There is $φ\in\Lom\infty$ such that 
for every $x_0>0$ there is $τ_2>0$ such that for every compact set $K\subset ℝ^{N-1}\times(x_0,∞)\times(0,τ_2)$ there are $ε_0>0$ and $C>0$ such that 
\[
 \ue(x,t)-\kkl{S_2(t)φ_b}(x) =  \ue(x,t) \ge C\sqrt{ε} 
\]
for every $ε\in(0,ε_0)$ and every $(x,t)\in K$.
\end{thm}

\begin{thm}\label{thm:lowerbound-ball}
 Let $\Om=\ballexterior$ and $φ_b\equiv 0$. Then there is $φ\in \Lom{\infty}$ such that 
 for every $r>1$ there is $τ_2>0$ such that for every compact set 
$K\subset (ℝ^3\setminus B_r(0) )\times(0,τ_2)$ there are $C>0$ and $ε_0>0$ such that 
\[
 \ue(x,t) - \kkl{S_2(t)φ_b} (x) = \ue(x,t) \ge C\sqrt{ε}
\]
for every $ε\in(0,ε_0)$ and every $(x,t)\in K$.
\end{thm}

The paper is organized as follows. In Section~\ref{sec:notation} we introduce some notation, in Section~\ref{sec:estimates-halfspace} we recall some estimates
in the case of the half-space, and in Section~\ref{sec:estimates-exterior} we derive analogous estimates for the exterior domain. In Section~\ref{sec:existence} 
we prove Theorem~\ref{thm:ball} and in Section~\ref{sec:upperbounds} 
Theorems~\ref{thm:upperbound-halfspace} and \ref{thm:upperbound-ball}. Section~\ref{sec:remark-condition} is devoted to a remark on the long-time behaviour
of solutions of the heat equation on the exterior domain and Sections~\ref{sec:lowerbound-halfspace} and \ref{sec:lowerbound:exterior} to the proofs of Theorems~\ref{thm:lowerbound-halfspace}
and \ref{thm:lowerbound-ball}, respectively.
\section{Preparation: Introducing further notation. The space \texorpdfstring{$X_T$}{$X$\textsubscript{$T$}}.} \label{sec:notation}

With the abbreviations 
\begin{equation*}
 \De {ψ} (x,t) := \int_0^t \kkl{S_1\kl{\nf{(t-s)}{ε}}F_1[ψ](\cdot,s)}(x) \ds, \qquad ψ\in \Ldom{∞},  
\end{equation*}
and 
\begin{equation}\label{defDet}
 \Det{ψ} (x,t) = \int_0^t \kkl{S_1\kl{\nf{(t-s)}{ε}}F_2[ψ](\cdot,s)}(x) \ds, \qquad ψ\in \Lom{∞},
\end{equation}
we introduce 
\begin{equation*}\label{defQe}
 \Qe [v](x,t):=\kkl{S_1\kl{\nf{t}{ε}}Φ}(x) - \De{φ_b}(x,t) -\Det v(x,t).
\end{equation*}

The proofs will be based on a fixed point argument for $\Qe$ in the space $X_T$, which we define as 
\begin{equation*}
 X_T:=\set{v\in C(\Ombar\times(0,T)) \mid ∂_{x_N}v\in C(\Ombar\times(0,T)),\;\;\norm[X_T]{v}<\infty}
\end{equation*}  
if $\Om=ℝ^N_+$, and as   
\begin{equation*}
X_T:=\set{v\in C(\Ombar\times(0,T)) \mid v \text{ radially symmetric}, ∂_{ν}v \in C(\Ombar\times(0,T)),\;\;\norm[X_T]{v}<\infty} 
\end{equation*}
for $\Om=\ballexterior$. In both cases, we equip it with the norm 
\begin{equation}\label{defnormXT}
 \norm[X_T]{v}:=\sup_{t\in(0,T)} E_{ε}[v](t),
\end{equation}
where 
\begin{equation}\label{defEe}
 E_{ε}[v](t):=\norm[\Lom\infty]{v(\cdot,t)}+\sqrt{\nf t{ε}}\norm[\Ldom\infty]{∂_{ν}v(\cdot,t)},
\end{equation}
and observe that it thereby becomes a Banach space.

\section{Estimates: Recalling the case \texorpdfstring{$\Omega=ℝ^N_+$}{$\Omega=ℝ$\textsuperscript{N}\textsubscript{$+$}}}\label{sec:estimates-halfspace}

In this section we recall the necessary estimates from \cite{fila_ishige_kawakami_largediffusionlimit} for $\Om=ℝ^N_+$. They have been used in the proof of
Theorem~\ref{thm:halfspace} in \cite{fila_ishige_kawakami_largediffusionlimit} and will be essential for our proof of Theorems \ref{thm:upperbound-halfspace} and \ref{thm:lowerbound-halfspace}. We will take care of the corresponding results for $\Om=\ballexterior$ -- and their proofs -- in the next section. 

\begin{lemma}
 Let $\Om=ℝ^N_+$. Then for every $ϕ\in\Lom\infty$, 
 \begin{equation}\label{eq:alt:S1}
  \sup_{t>0} \norm[\Lom\infty]{S_1(t)ϕ}\le \norm[\Lom\infty]{ϕ}
 \end{equation}
 and 
 \begin{equation}\label{eq:alt:dNS1}
  \sup_{t>0} t^{\f12}\norm[\Lom\infty]{∂_{x_N}[S_1(t)ϕ]}\le \norm[\Lom\infty]{ϕ}.
 \end{equation}
 Moreover, for every $L>0$ there is $C>0$ such that for every $ϕ\in\Lom\infty$, $x\in ℝ^{N-1}\times(0,L)$ and $t>0$, 
 \begin{equation}\label{S1halfspacesmallness}
  |[S_1(t)ϕ](x)|\le Ct^{-\f12} \norm[\Lom\infty]{ϕ}
 \end{equation}
\end{lemma}
\begin{proof} 
\cite[(2.1) and (2.2) and proof of (2.3), respectively]{fila_ishige_kawakami_largediffusionlimit}. 
\end{proof}

\begin{lemma}\label{S2boundhalfspace}
 Let $\Om=ℝ^N_+$ and $ψ\in\Ldom{∞}$. Then for every $t\ge 0$ 
 \[
  \norm[\Lom{∞}]{S_2(t)ψ} \le \norm[\Ldom{∞}]{ψ}.
 \]
\end{lemma}
\begin{proof}
 \cite[(2.8)]{fila_ishige_kawakami_largediffusionlimit}.
\end{proof}

\begin{lemma}
 Let $\Om=ℝ^N_+$.
 There is $C>0$ such that 
 \begin{equation}\label{eq:alt:De}
  \norm[\Lom\infty]{\De{ψ}(\cdot,t)}\le C t^{\f14}\kl{ε^{\f12}+t^{\f34}}\norm[\Ldom\infty]{ψ}
 \end{equation}
 for all $ψ\in\Ldom\infty$, $ε>0$ and $t>0$.
 
 There is $C>0$ such that 
 \begin{equation}\label{eq:alt:dNDe}
  \norm[\Lom\infty]{∂_{x_N}\De{ψ}(\cdot,t)}\le Cε^{\f34}t^{-\f14}\norm[\Ldom\infty]{ψ}
 \end{equation}
 for all $ψ\in\Ldom\infty$, $ε>0$, $t>0$. 
 
 For every $L>0$ there is $C>0$ such that 
 \begin{equation}\label{Dehalfspacesmallness}
  \De{ψ}(x,t) \le C\norm[\Ldom\infty]{ψ}\sqrt{ε}(t^{\f14}+t^{\f12}) 
 \end{equation}
 for every $ψ\in \Ldom\infty$, $ε>0$, $t>0$, $x\inℝ^{N-1}\times(0,L)$. 
\end{lemma}
\begin{proof}
 \cite[(2.10), Lemma 2.4 and proof of (2.11), respectively]{fila_ishige_kawakami_largediffusionlimit}
\end{proof}

\begin{lemma}\label{Dethalfspace}
 Let $\Om=ℝ^N_+$. Then there is $C>0$ such that for every $T>0$, $v\in X_T$, $x\in\Om$, $t\in(0,T)$ 
 \begin{equation}\label{Dethalfspacesmallness}
  |\Det v (x,t)| \le Cε^{\f12}\norm[X_T]{v}(t^{\f12}+ε^{\f12}t^{\f34}+t)
 \end{equation}
 and 
 \[
  |∂_{x_N}\Det v (x,t)| \le C \sqrt{\nf{t}{ε}}\norm[X_T]{v} \kl{(εt)^{\f12}+(εt)^{\f78}+(εt)^{\f34}}.
 \]
\end{lemma}
\begin{proof}
 \cite[proof of (3.9) and of (3.10), respectively]{fila_ishige_kawakami_largediffusionlimit}
\end{proof}

\section{Estimates: \texorpdfstring{$ℝ^3\setminus B_1(0)$}{$ℝ$\textsuperscript{$3$}\textbackslash $B_1$($0$)}}\label{sec:estimates-exterior}
The assertions of the lemmata in this section parallel those in the previous section. Before we begin dealing with their statements and proofs, let us first bring some of the quantities that have been defined in the introduction in the explicit form in which the radially symmetric $3$-dimensional setting allows us to express them.

If $ψ$ is a radially symmetric function on $∂B_1(0)$, we can interpret it as a real number and write 
\[
 \kkl{S_2(t)ψ}(r) = \f{ψ}{re^t}, \qquad r>1, t\ge 0.
\]
This means that also 
\[
 F_1[φ_b](r,t)=-\f{φ_b}{re^t},
\]
\[
 F_2[v](r,t)= \f{v_r(1,t)}{r} + \int_0^t \f{v_r(1,s)}{re^{t-s}} \ds,
\]
and 
\[
 Φ(r)=φ(r)-\f{φ_b}{r}.
\]


Also $S_1$ can be written explicitly: 

\[
 \kkl{S_1(t)φ}(r) = \f1{r\sqrt{4πt}} \int_1^{∞} \kl{e^{-\f{(r-ρ)^2}{4t}}-e^{-\f{(r+ρ-2)^2}{4t}}}ρφ(ρ)\drho , \qquad r>1, t>0.
\]
This representation is, of course, based on the fact that for every radially symmetric solution $u$ of the heat equation in $ℝ^3$, the function $(r,t)\mapsto ru(r,t)$ solves the one-dimensional heat equation.

\begin{lemma}
 Let $\Om=\ballexterior$. Then for every radially symmetric $ϕ\in\Lom\infty$ 
 \begin{equation}\label{eq:S1ballexterior}
  \sup_{t>0} \norm[\Lom\infty]{S_1(t)ϕ}\le \norm[\Lom\infty]{ϕ}.
 \end{equation}
 If, moreover, $\sup_{ρ\ge 1} |ρφ(ρ)|\le K$ for $K>0$, then 
 \begin{equation}\label{S1ballsmall}
  |\kkl{S_1(t)φ}(r)|\le \f{(r-1)K}{r\sqrt{πt}}\qquad \text{for every } r>1, t>0,
 \end{equation}
 and 
 \begin{equation}\label{eq:drS1ballexterior-decreasing}
  |∂_r\kkl{S_1(t)φ}(r)|\le \f{2K}{\sqrt{πt}} \qquad \text{for every } r>1, t>0.
 \end{equation}
\end{lemma}

\begin{proof}
We obtain \eqref{eq:S1ballexterior} from explicit computations as follows: 
\begin{align*}
 |\kkl{S_1(t)φ}(r)| &= \left|\f1{r\sqrt{4πt}} \int_1^{∞} \kl{e^{-\f{(r-ρ)^2}{4t}}-e^{-\f{(r+ρ-2)^2}{4t}}}ρφ(ρ)\drho \right|\\
&\le \f{\norm[\Lom\infty]{φ}}{r\sqrt{4πt}} \int_1^{∞} \kl{e^{-\f{(r-ρ)^2}{4t}}-e^{-\f{(r+ρ-2)^2}{4t}}} ρ \drho  \\
&= \f{\norm[\Lom\infty]{φ}}{r\sqrt{π}} \kl{\int_{\f{1-r}{2\sqrt t}}^{∞} (2\sqrt{t}z+r)e^{-z^2} \dz - \int_{\f{r-1}{2\sqrt{t}}}^{∞} (2\sqrt t z +2-r)e^{-z^2}\dz }\\
 &= \f{2\sqrt t}{r\sqrt{π}} \norm[\Lom\infty]{φ} \int_{\f{1-r}{2\sqrt t}}^{\f{r-1}{2\sqrt t}} ze^{-z^2} \dz 
 +\f{r\norm[\Lom\infty]{φ}}{r\sqrt{π}}\int_{\f{1-r}{2\sqrt t}}^{\infty} e^{-z^2}\dz\\
 &\quad + \f{(r-2)\norm[\Lom\infty]{φ}}{r\sqrt{π}}\int_{\f{r-1}{2\sqrt t}}^{\infty} e^{-z^2} \dz \\
&= \kl{1-\f2{r\sqrt{π}}\int_{\f{r-1}{2\sqrt t}}^{∞}e^{-z^2}\dz}\norm[\Lom{∞}]{φ}, \qquad r\ge 1, t>0.
\end{align*}
If we not only control $\norm[\Lom\infty]{φ}$, but even $\sup_{ρ\ge1}|ρφ(ρ)|$, we can proceed slightly differently: 
\begin{align}\label{estimateS1-decaying}
 |\kkl{S_1(t)φ}(r)| &= \left|\f1{r\sqrt{4πt}} \int_1^{∞} \kl{e^{-\f{(r-ρ)^2}{4t}}-e^{-\f{(r+ρ-2)^2}{4t}}}ρφ(ρ)\drho \right|\nn\\
&\le \f{\sup_{ρ\ge 1} |ρφ(ρ)|}{r\sqrt{4πt}} \int_1^{∞} \kl{e^{-\f{(r-ρ)^2}{4t}}-e^{-\f{(r+ρ-2)^2}{4t}}} \drho \nn\\
&= \f{\sup_{ρ\ge 1} |ρφ(ρ)|}{r\sqrt{π}} \kl{\int_{\f{1-r}{2\sqrt t}}^{∞} e^{-z^2} \dz - \int_{\f{r-1}{2\sqrt t}}^{∞} e^{-z^2} \dz}\nn\\
&= \f{2\sup_{ρ\ge 1} |ρφ(ρ)|}{r\sqrt{π}}\int_0^{\f{r-1}{2\sqrt{t}}} e^{-z^2}\dz 
\le \f{(r-1)\sup_{ρ\ge 1} |ρφ(ρ)|}{r\sqrt{πt}}
\end{align}
holds for every $r>1, t>0$ and \eqref{S1ballsmall} follows. 

For the estimate of the radial derivative let us first observe that for every $t>0$ and $r\ge 1$
\begin{align}\label{eq:drS1anfang}
 &∂_r\kkl{S_1(t)φ}(r)
 = -\f1{r^2\sqrt{4πt}}\int_1^{∞} \kl{e^{-\f{(r-ρ)^2}{4t}}-e^{-\f{(r+ρ-2)^2}{4t}}}ρφ(ρ)\drho \nn\\
 &\quad+ \f1{r\sqrt{4πt}} \int_1^{∞} \kl{-\f{r-ρ}{2t} e^{-\f{(r-ρ)^2}{4t}} + \f{r+ρ-2}{2t}e^{-\f{(r+ρ-2)^2}{4t}}}ρφ(ρ)\drho \\
 & = -\f1r \kkl{S_1(t)φ}(r) + \f1{r\sqrt{4πt}} \int_1^{∞} \kl{-\f{r-ρ}{2t} e^{-\f{(r-ρ)^2}{4t}} + \f{r+ρ-2}{2t}e^{-\f{(r+ρ-2)^2}{4t}}}ρφ(ρ)\drho. \nn
\end{align}
We see that here for every $t>0$ and $r\ge 1$
%
%
%
%
%

\begin{align}\label{drS1decaying}
  \qquad &\hspace{-2.5cm}\betrag{\int_1^{∞} \kl{-\f{r-ρ}{2t} e^{-\f{(r-ρ)^2}{4t}} 
+ \f{r+ρ-2}{2t}e^{-\f{(r+ρ-2)^2}{4t}}}ρφ(ρ)\drho  }\nn\\ 
 &\le \kl{2\int_{\f{1-r}{2\sqrt t}}^{∞} |z|e^{-z^2}  \dz +
 +2\int_{\f{r-1}{2\sqrt{t}}}^{∞} ze^{-z^2} \dz} \sup_{ρ\ge1} |ρφ(ρ)| \nn\\
 &= 2\int_{ℝ} |z|e^{-z^2} \dz \sup_{ρ\ge1} |ρφ(ρ)| = 2 \sup_{ρ\ge1} |ρφ(ρ)|
\end{align}

Hence, inserting 
\eqref{estimateS1-decaying} 
and \eqref{drS1decaying} into \eqref{eq:drS1anfang}, we arrive at 
\[
  |∂_r\kkl{S_1(t)φ}(r)|\le \f{(r-1)\sup_{ρ\ge 1} |ρφ(ρ)|}{r^2\sqrt{πt}} + \f{2 \sup_{ρ\ge 1} |ρφ(ρ)|}{r\sqrt{4πt}}  \le \f{2}{\sqrt{πt}}\sup_{ρ\ge 1} |ρφ(ρ)|
 \]
for every $r\ge 1, t>0$.
\end{proof}

\begin{lemma}\label{S2boundball}
 Let $\Om=\ballexterior$ and let $ψ\in \Ldom{∞}$ be radially symmetric (constant). Then for every $t\ge 0$,  
 \[
  \norm[\Lom{∞}]{S_2(t)ψ} \le \norm[\Ldom{∞}]{ψ}.
 \]
\end{lemma}
\begin{proof}
 This is obvious from the explicit form of $S_2(t)ψ = \f{ψ}{re^t}$, $r\ge 1, t\ge 0$.
\end{proof}

\begin{lemma}\label{Deball}
 Let $\Om=\ballexterior$. Then there is $C>0$ such that for every $ψ\in ℝ$ and every $t>0$, $ε>0$, 
 \begin{equation}\label{eq:Deball}
  \norm[\Lom\infty]{\De{ψ}(\cdot,t)} \le C \sqrt{εt} |ψ|
 \end{equation}
 and 
 \begin{equation}\label{eq:drDeball}
  \norm[\Lom\infty]{∂_r \De{ψ}(\cdot,t)} \le C \sqrt{εt} |ψ|.
 \end{equation}
\end{lemma}

\begin{proof}
We use the explicit representation for $\De{ψ}$ to see that for every $r>1$, $t>0$, $ε>0$, 
\begin{align*}
 \De{ψ} (r,t) &= -\int_0^t S_1\kl{\nf{(t-s)}{ε}}\f{ψ}{re^s} \ds \\
  &=-\int_0^t \f{1}{r\sqrt{4π\f{t-s}{ε}}}\int_1^{∞} \kl{e^{-\f{(r-ρ)^2}{4\f{t-s}{ε}}}-e^{-\f{(r+ρ-2)^2}{4\f{t-s}{ε}}}} ρ \f{ψ}{ρe^s} \drho \ds\\
  &=-\f{ψ}{r}\sqrt{\f{ε}{4π}}\int_0^t \f{e^{-s}}{\sqrt{t-s}} \int_1^{∞} \kl{e^{-\f{(r-ρ)^2}{4\f{t-s}{ε}}}-e^{-\f{(r+ρ-2)^2}{4\f{t-s}{ε}}}} \drho \ds \\
  &= -\f{ψ}{r}\sqrt{\f{ε}{4π}}\int_0^t \f{e^{-s}}{\sqrt{t-s}}\sqrt{4\f{t-s}{ε}}\kl{ \int_{\f{1-r}{\sqrt{4\f{t-s}{ε}}}}^{∞} e^{-z^2}\dz - \int_{\f{1+r-2}{\sqrt{4\f{t-s}{ε}}}}^{∞} e^{-z^2}\dz }\\
  &= - \f{2ψ}{r\sqrt{π}}\int_0^t e^{-s} \int_0^{\f{(r-1)\sqrt{ε}}{2\sqrt{t-s}}} e^{-z^2} \dz \ds. 
\end{align*}
Due to the elementary estimates $\int_0^L e^{-z^2} \dz\le L$ for every $L>0$ and $\int_0^t e^{-s}(t-s)^{-\f12} \ds\le 2\sqrt{t}$ for every $t>0$, this implies \eqref{eq:Deball}. 

Concerning the radial derivative, for every $r>1$, $t>0$, $ε>0$, we have 
\begin{align*}
 ∂_r \De{ψ} (r,t) &= \f{2ψ}{r^2\sqrt{π}}\int_0^t e^{-s} \int_0^{\f{(r-1)\sqrt{ε}}{2\sqrt{t-s}}} e^{-z^2} \dz \ds
 - \f{2ψ}{r\sqrt{π}}\int_0^t e^{-s} 
  \f{\sqrt{ε}}{2\sqrt{t-s}} e^{-\f{(r-1)^2ε}{4(t-s)}}
 \ds.
\end{align*}
 Here, due to $\f1r\le1$, the first term can be estimated by \eqref{eq:Deball}; the second obeys 
 \[
  \left\lvert- \f{2ψ}{r\sqrt{π}}\int_0^t e^{-s} 
  \f{\sqrt{ε}}{2\sqrt{t-s}} e^{-\f{(r-1)^2ε}{4(t-s)}}
 \ds\right\rvert \le\f{|ψ|\sqrt{ε}}{r\sqrt{π}}\int_0^t e^{-s}(t-s)^{-\f12}\ds \le \f{2|ψ|\sqrt{ε}\sqrt{t}}{r\sqrt{π}} 
 \]
 for every $t>0$, $r>1$, $ε>0$. 
\end{proof}

\begin{lemma}
 Let $\Om=\ballexterior$. Then there is $C>0$ such that for every $T>0$ $t\in(0,T)$, $ε>0$ and every radial function $v\colon \Ombar\times(0,T)\toℝ$ with $∂_{ν}v\bdry\in\Ldom\infty$, 
 \begin{equation}\label{eq:Detball}
  \norm[\Lom\infty]{\Det v(\cdot,t)} \le \sqrt{επ} (1+t) \sup_{τ\in (0,T)} |\sqrt{τ} ∂_r v(1,τ)|
 \end{equation}
 and 
 \begin{equation}\label{eq:drDetball}
\norm[\Lom\infty]{∂_r\Det v(\cdot,t)} \le (1+2\sqrt{t}+2t) \sqrt{επ} \sup_{τ\in (0,T)} |\sqrt{τ} ∂_r v(1,τ)|  
 \end{equation}
\end{lemma}
\begin{proof}
If we insert the explicit definitions of $S_1$ and $S_2$ into \eqref{defDet}, we see that for every $r\ge 1, t>0, ε>0$ we have  
\begin{align*}
 \Det{v}(r,t) &= \int_0^t\int_1^{∞} \f{1}{r\sqrt{4π\f{t-s}{ε}}}\kl{e^{-\f{(r-ρ)^2}{4\f{t-s}{ε}}} - e^{-\f{(r+ρ-2)^2}{4\f{t-s}{ε}}}}ρ \f{∂_r v(1,s)}{ρ} \drho \ds \\
 &\qquad - \int_0^t \int_1^{∞}\int_0^s \f1{r\sqrt{4π\f{t-s}{ε}}} \kl{e^{-\f{(r-ρ)^2}{4\f{t-s}{ε}}} - e^{-\f{(r+ρ-2)^2}{4\f{t-s}{ε}}}}ρ \f{∂_r v(1,σ)}{ρe^{s-σ}} \dsigma \drho \ds \\
 &= \f2{r\sqrt{π}}\int_0^t ∂_rv(1,s)\int_0^{\f{r-1}{\sqrt{4\f{t-s}{ε}}}} e^{-z^2} \dz\ds \\
 &\qquad- \f2{r\sqrt{π}}\int_0^t\int_0^s e^{σ-s}∂_rv(1,σ) \dsigma  \int_0^{\f{r-1}{\sqrt{4\f{t-s}{ε}}}} e^{-z^2} \dz \ds, 
\end{align*}
so that 
\begin{align}\label{eq:Detanfang} 
& \betrag{\Det{v}(r,t)}
  \le \f2{r\sqrt{π}}\int_0^t \f{r-1}{\sqrt{s}\sqrt{4\f{t-s}{ε}}} \ds \sup_{s\in (0,t)} |\sqrt{s} ∂_rv(1,s)|\nn \\
 &\qquad + \f2{r\sqrt{π}}\int_0^t\int_0^s \f{e^{σ-s}}{\sqrt{σ}} \dsigma  \f{r-1}{\sqrt{4\f{t-s}{ε}}} \ds 
\sup_{σ\in (0,t)} |\sqrt{σ}∂_rv(1,σ)|\\
 &=\f{(r-1)\sqrt{ε}}{r\sqrt{π}}\kl{\int_0^t \f{1}{\sqrt{s}\sqrt{t-s}} \ds + \int_0^t\int_0^s \f{e^{σ-s}}{\sqrt{σ}} \dsigma  \f{1}{\sqrt{t-s}} \ds} \sup_{s\in (0,t)} |\sqrt{s}∂_rv(1,s)|\nn
\end{align}
Taking into account that 
\begin{equation}\label{pione}
 \int_0^t s^{-\f12}(t-s)^{-\f12} \ds = \int_0^1 σ^{-\f12}(1-σ)^{-\f12} \dsigma  =π \qquad \text{for all } t>0 
\end{equation}
and 
\begin{equation}\label{pitwo}
 \int_0^t \int_0^s \f{e^{σ-s}}{\sqrt{σ}}\dsigma  (t-s)^{-\f12} \ds \le \int_0^t \f{2\sqrt{s}}{\sqrt{t-s}} \ds = 2t \int_0^1\sqrt{\f{σ}{1-σ}}\dsigma  = πt \qquad \text{for all } t>0, 
\end{equation}
we see that \eqref{eq:Detanfang} proves 
%
%
%
\eqref{eq:Detball}.

As to the radial derivative, we compute 
\begin{align*}
 ∂_r &\Det{v} (r,t)\\
 &= -\f2{r^2\sqrt{π}}\int_0^t ∂_rv(1,s)\int_0^{\f{r-1}{\sqrt{4\f{t-s}{ε}}}} e^{-z^2} \dz\ds
 +  \f1{r\sqrt{π}}\int_0^t ∂_rv(1,s)
 \f{\sqrt{ε}}{\sqrt{t-s}} e^{-\f{(r-1)^2ε}{4(t-s)}} \ds \\ 
 &\qquad + \f2{r^2\sqrt{π}}\int_0^t\int_0^s e^{σ-s}∂_rv(1,σ) \dsigma  \int_0^{\f{r-1}{\sqrt{4\f{t-s}{ε}}}} e^{-z^2} \dz \ds \\
&\qquad - \f1{r\sqrt{π}}\int_0^t\int_0^s e^{σ-s}∂_rv(1,σ) \dsigma  \f{\sqrt{ε}}{\sqrt{t-s}} e^{-\f{(r-1)^2ε}{4(t-s)}}  \ds \qquad\text{ for } r\ge 1, t>0. 
 \end{align*}
Here, we can simplify two integrals according to 
\begin{align*}
 &-\f2{r^2\sqrt{π}}\int_0^t ∂_rv(1,s)\int_0^{\f{r-1}{\sqrt{4\f{t-s}{ε}}}} e^{-z^2} \dz\ds\\
 &\quad + \f2{r^2\sqrt{π}}\int_0^t\int_0^s e^{σ-s}∂_rv(1,σ) \dsigma  \int_0^{\f{r-1}{\sqrt{4\f{t-s}{ε}}}} e^{-z^2} \dz \ds
= -\f1r \Det{v}(r,t),\quad  r\ge 1, t>0.
\end{align*}
The next one can be estimated as 
\begin{align}\label{estdrei}
& \betrag{\f1{r\sqrt{π}}\int_0^t ∂_rv(1,s)
 \f{\sqrt{ε}}{\sqrt{t-s}} e^{-\f{(r-1)^2ε}{4(t-s)}} \ds}\le \f{\sqrt{ε}}{r\sqrt{π}} \sup_{s\in(0,t)}|\sqrt{s} ∂_r v(1,s)| \int_0^t s^{-\f12} (t-s)^{-\f12} \ds\nn\\
 &\le \f{2\sqrt{επt}}{r} \sup_{s\in (0,t)}|\sqrt{s} ∂_r v(1,s)|, \qquad r\ge 1, t>0, 
\end{align}
according to \eqref{pione}, and for the last we again rely on \eqref{pitwo} to see that  
\begin{align}\label{estvier}
&\betrag{- \f1{r\sqrt{π}}\int_0^t\int_0^s e^{σ-s}∂_rv(1,σ) \dsigma  \f{\sqrt{ε}}{\sqrt{t-s}} e^{-\f{(r-1)^2ε}{4(t-s)}}  \ds  }\nn\\
& \le \f{\sqrt{ε}}{r\sqrt{π}}\int_0^t \sup_{σ\in(0,t)}|\sqrt{σ}∂_rv(1,σ)| \int_0^s \f{e^{σ-s}}{\sqrt{σ}} \dsigma  (t-s)^{-\f12} \ds\nn\\
& \le \f{t\sqrt{επ}}{r} \sup_{s\in (0,t)}|\sqrt{s}∂_rv(1,s)|\qquad \text{for all } r\ge 1, t>0.
\end{align}
If we finally combine \eqref{eq:Detball}, \eqref{estdrei} and \eqref{estvier}, we can readily conclude \eqref{eq:drDetball}.
\end{proof}

\begin{lemma}\label{lem:Det-contraction}
 Let $\Om=\ballexterior$ and $T>0$. Then for every $v\in X_T$, 
\begin{equation}\label{ballDetsmall}
 \norm[X_T]{\Det v} \le (ε (1+T) + \sqrt{ε}(\sqrt T+2T+2T\sqrt{T} ))\sqrt{π}\norm[X_T]{v}.  
\end{equation}
\end{lemma}

\begin{proof}
 If we insert \eqref{eq:Detball} and \eqref{eq:drDetball} into \eqref{defnormXT}, we immediately obtain 
 \begin{align*}
  \norm[X_T]{\Det v} &= \sup_{t\in(0,T)} \kl{\norm[\Lom\infty]{\Det v(\cdot,t)}+\sqrt{\nf t{ε}} \norm[\Lom\infty]{∂_r \Det v(\cdot,t)}}\\
  &\le \sup_{t\in(0,T)} \kl{\kl{\sqrt{επ}(1+t)+\sqrt{\nf{t}{ε}}(1+2\sqrt{t}+2t)\sqrt{επ} } \sup_{s\in(0,t)} |\sqrt{s}∂_rv(1,s)| }\\
  &\le \sup_{t\in(0,T)} \kl{\sqrt{επ}(1+t) + (\sqrt{t}+2t+2t\sqrt{t})\sqrt{π}} \sqrt{ε}\norm[X_T]{v}\\
  &\le (ε (1+T) + \sqrt{ε}(\sqrt T+2T+2T\sqrt{T}) )\sqrt{π}\norm[X_T]{v}\qedhere
 \end{align*}
\end{proof}

\section{Existence. Proof of Theorem \texorpdfstring{\ref{thm:ball}}{3}}\label{sec:existence}
\begin{proof}[Proof of Theorem \ref{thm:ball}]
\textbf{a)} 
The explicit representation of $S_1\kl{\nf t{ε}}\Phi$, $\De{φ_b}$ and $\Det{v}$ and their derivatives show that (for each $ε$, $T$, radially symmetric $φ\in\Lom\infty$ and $φ_b\in\Ldom\infty$) $\Qe$ maps $X_T$ into $X_T$.\\
Moreover, $\Qe$ is a contraction: 
Let $ε_0\in(0,\nf1{\sqrt{π}})$, let $T_*>0$ be so small that  $(ε_0 (1+T_*) + \sqrt{ε_0}(\sqrt T_*+2T_*+2T_*\sqrt{T_*} )\sqrt{π}=:q<1$. 
Then, according to Lemma \ref{lem:Det-contraction}, for every $ε\in(0,ε_0)$, every $T\in(0,T_*)$, and every $v_1,v_2\in X_T$,  
\[
 \norm[X_T]{\Qe[v_1]-\Qe[v_2]} = \norm[X_T]{\Det {v_1-v_2}} \le q \norm[X_T]{v_1-v_2}. 
\]
Banach's fixed point theorem hence yields a unique solution $\ve$ in $X_T$ (for $T\in(0,T_*)$), and $\we$ is defined according to \eqref{defwe}. Since $T_*$ was independent of the initial data, successive application of the same reasoning finally provides a global-in-time solution. 

\textbf{b)} 
For the second part of the theorem -- and, in particular, for more quantitative and $ε$-independent information, which by way of Lemma \ref{lem:Det-contraction} will rely on a uniform bound for $\norm[X_T]{\ve}$ --
the mere existence result obtained above is insufficient.

We hence restrict the class of admissible initial data 
and will attempt to apply the fixed point theorem not on $X_T$, but on a bounded subset thereof. 

Let $ε_0$, $T_*$ and $q$ be as before. 
Let 
\[
M>\f1{1-q}\kl{\norm[\Lom\infty]{φ}+\kl{1+\f2{\sqrt{π}}+2C} |φ_b|+\f{2}{\sqrt{π}} \sup_{ρ\ge1} |ρφ(ρ)| },
\]
where $C$ is the constant in Lemma \ref{Deball}.

If we can show that for every $T\in(0,\min\set{1,T_*}), ε\in(0,ε_0)$, $\Qe$ maps the set 
\[B_M:=\set{v\in X_T\mid \norm[X_T]{v}\le M}\]
into itself, Banach's fixed point theorem does not only prove existence of a solution $\ve$, but also shows 
\[
 \norm[X_T]{\ve}\le M,
\]
which finally proves \eqref{thm:ball-estimate}. 

We therefore turn our attention to the proof of $\Qe(B_M)\subset B_M$: 
Due to \eqref{defEe}, \eqref{eq:S1ballexterior} and \eqref{eq:drS1ballexterior-decreasing}, we have 
 \begin{align*}
  \Ee[S_1\kl{\nf{\cdot}{ε}}Φ] (t) &= \norm[\Lom\infty]{S_1\kl{\nf{t}{ε}}Φ} +\sqrt{\nf{t}{ε}}\norm[\Lom\infty]{∂_rS_1\kl{\nf{t}{ε}}Φ}\\
  &\le \norm[\Lom\infty]{Φ} + \sqrt{\nf{t}{ε}}\f{2}{\sqrt{π}}\sqrt{\nf{ε}{t}} \sup_{ρ\ge 1} |ρΦ(ρ)| \qquad \text{for every } t>0, ε>0, 
 \end{align*}
 where 
\[
 \sup_{ρ\ge 1} |ρΦ(ρ)|= \sup_{ρ\ge 1} \betrag{ρφ(ρ) - ρ\f{φ_b(ρ)}{ρ}} \le \sup_{ρ\ge 1} |ρφ(ρ)| + |φ_b|
\]
and
\[
 \norm[\Lom\infty]{Φ}\le \norm[\Lom\infty]{φ}+\sup_{ρ\ge1} \betrag{\f{φ_b}{ρ}} \le \norm[\Lom{∞}]{φ} + |φ_b|.
\]

If we furthermore insert \eqref{eq:Deball} and \eqref{eq:drDeball} into \eqref{defnormXT}, we see that with $C$ from Lemma \ref{Deball} for any $ε\in(0,1)$, $t\in(0,1)$ we have 
\begin{align*}
& \Ee[\De{φ_b}](t) = \norm[\Lom{∞}]{\De{φ_b}(\cdot,t)} + \sqrt{\nf{t}{ε}}\norm[\Lom{∞}]{∂_r\De{φ_b}(\cdot,t)}\\
 &\le C\sqrt{εt}|φ_b| + \sqrt{\nf{t}{ε}}C\sqrt{εt}|φ_b|
 \le C(\sqrt{εt}+t)|φ_b|
 \le 2C|φ_b|.
\end{align*}
Finally, the choice of $q$ and Lemma \ref{lem:Det-contraction} ensure 
\begin{align*}
 \norm[X_T]{\Det v} \le q \norm[X_T]{v}.
\end{align*}

In conclusion, for $ε\in(0,ε_0)$, $T\in(0,\min\set{1,T_*})$ and every $v\in X_T$, 
\begin{align*}
 \norm[X_T]{\Qe v} & \le \norm[\Lom{∞}]{φ} + |φ_b| + \f2{\sqrt{π}} {\sup_{ρ\ge 1} |ρφ(ρ)| + |φ_b|} + 2C|φ_b| +q\norm[X_T]{v} \\
&  \le (1-q)M + q\norm[X_T]{v}, 
\end{align*}
so that $\norm[X_T]{\Qe v}\le M$ whenever 
$\norm[X_T]{v}\le M$. 

We postpone the proof of \eqref{thm:ball-eqlimit}, seeing that it is a corollary of
Theorem~\ref{thm:upperbound-ball}. 
\end{proof}

\begin{remark}\label{remark:existence-halfspace}
The proof of the existence result in Theorem~\ref{thm:halfspace} in \cite{fila_ishige_kawakami_largediffusionlimit} proceeds analogously. 
In place of \eqref{eq:S1ballexterior} and \eqref{eq:drS1ballexterior-decreasing}, \eqref{eq:Deball}, \eqref{eq:drDeball} and Lemma \ref{lem:Det-contraction}, one has to use \eqref{eq:alt:S1}, \eqref{eq:alt:dNS1}, \eqref{eq:alt:De}, \eqref{eq:alt:dNDe} and Lemma \ref{Dethalfspace}.
\end{remark}

\section{Upper bounds. Proofs of Theorems \texorpdfstring{\ref{thm:upperbound-halfspace}}{4} and \texorpdfstring{\ref{thm:upperbound-ball}}{5}.} \label{sec:upperbounds}
The estimates we have given in Sections~\ref{sec:estimates-halfspace} and \ref{sec:estimates-exterior} mainly deal with $\ve$. In preparation of the proofs of the estimates of $\ue=\ve+\we$ let us state the following lemma, which for $\Om=ℝ^N_+$ also was part of \cite[proof of Theorem~1.1~(c)]{fila_ishige_kawakami_largediffusionlimit}: 
\begin{lemma}\label{lem:weminusS2}
 Let $\Om=ℝ^N_+$, $N\ge 2$, or $\Om=\ballexterior$, and let $t>0$, $ε>0$. Then for every $\ve\in X_t$ and $\we$ as in \eqref{defwe} we have 
 \begin{equation}\label{eq:weminusS2}
  \norm[\Lom\infty]{\we(\cdot,t)-S_2(t)φ_b}\le 2\sqrt{εt} \norm[X_t]{\ve}.
 \end{equation}
\end{lemma}
\begin{proof}
According to \eqref{defwe}, Lemma \ref{S2boundhalfspace} (for $\Om=ℝ^N_+$) or Lemma \ref{S2boundball} (for $\Om=\ballexterior$) and \eqref{defnormXT}, we can estimate 
\begin{align*}
& \norm[\Lom\infty]{\we(\cdot,t)-S_2(t)φ_b}=\norm[\Lom\infty]{\int_0^t S_2(t-s)∂_{ν} \ve(\cdot,s)\bdry \ds}\\
& \le \int_0^t\norm[\Ldom\infty]{∂_{ν}\ve(\cdot,s)}\ds
\le \norm[X_t]{\ve}\sqrt{ε}\int_0^t \f1{\sqrt{s}} \ds = 2\sqrt{εt} \norm[X_t]{\ve}. \qedhere
\end{align*} 
\end{proof}

\begin{proof}[Proof of Theorem \ref{thm:upperbound-halfspace}]
We let $K\subset \Ombar$ be compact. Then by the definition of $\ue$ (cf. \eqref{defQe} and Definition \ref{def:solution}, or proof of Theorem~\ref{thm:ball} and Remark \ref{remark:existence-halfspace}) we have 
\begin{align}\label{firstestimateinproofofthmupperboundhalfspace}
 \norm[L^{∞}(K)]{\ue(\cdot,t)-S_2(t)φ_b} &=  \norm[L^{∞}(K)]{\ve(\cdot, t) + \we(\cdot,t)-S_2(t)φ_b}\nn\\
 &\le \norm[L^{∞}(K)]{\ve(\cdot,t)} + \norm[L^{∞}(K)]{\we(\cdot,t)-S_2(t)φ_b}\nn\\
 &\le \norm[L^{∞}(K)]{S_1\kl{\nf{t}{ε}}Φ} + \norm[L^{∞}(K)]{\De{φ_b}(\cdot,t)} \nn\\
&\quad +\norm[L^{∞}(\Om)]{\Det \ve(\cdot,t)} +\norm[L^{∞}(\Om)]{\we(\cdot,t)-S_2(t)φ_b}
\end{align}
for every $t>0$. Here we can apply \eqref{S1halfspacesmallness}, \eqref{Dehalfspacesmallness}, 
\eqref{Dethalfspacesmallness} and \eqref{eq:weminusS2} so as to obtain $C_1=C_1(K)>0$, $C_2=C_2(K)>0$ and $C_3>0$ such that  
\begin{align*}
&\norm[L^{∞}(K)]{S_1\kl{\nf{t}{ε}}Φ} + \norm[L^{∞}(K)]{\De{φ_b}(\cdot,t)} + \norm[L^{∞}(\Om)]{\Det \ve(\cdot,t)}\\
&\quad +\norm[L^{∞}(\Om)]{\we(\cdot,t)-S_2(t)φ_b}
 \le  C_1 \sqrt{\nf{ε}t} \norm[\Lom\infty]{Φ} + C_2\norm[\Ldom\infty]{φ_b}\sqrt{ε}(t^{\f14}+t^{\f12})\\
&\quad +C_3ε^{\f12}\norm[X_t]{v}(t^{\f12}+ε^{\f12}t^{\f34}+t)+2\sqrt{εt}\norm[X_t]{\ve}
\end{align*}
for every $t>0$ and $ε\in(0,1)$. If we take the supremum over $t\in(τ_1,τ_2)$ and use boundedness of $\norm[X_{τ_2}]{\ve}$ according to \eqref{halfspacenormXTfinite}, we thereby have proven
Theorem~\ref{thm:upperbound-halfspace}.
\end{proof}

\begin{proof}[Proof of Theorem \ref{thm:upperbound-ball}]
Similarly to \eqref{firstestimateinproofofthmupperboundhalfspace}, 
the construction of $\ue$ in the proof of
Theorem~\ref{thm:ball} (or directly Definition \ref{def:solution}) allows us to estimate 
\begin{align*}
\norm[\Lom{∞}]{\ue(\cdot,t)-S_2(t)φ_b} \le ~& \norm[\Lom{∞}]{S_1\kl{\nf{t}{ε}}Φ} 
 + \norm[\Lom{∞}]{\De{φ_b}(\cdot,t)} \\ &+ \norm[\Lom{∞}]{\Det v(\cdot,t)} +
\norm[\Lom{∞}]{\we(\cdot,t)-S_2(t)φ_b}\nn
\end{align*}
for $ε\in(0,ε_0)$ and $t>0$. If we abbreviate $K:=\sup_{ρ\ge1} |ρΦ(ρ)|$ and employ \eqref{S1ballsmall}, 
\eqref{eq:Deball}, \eqref{ballDetsmall}, \eqref{eq:weminusS2}, we obtain $C>0$ such that 
the last inequality turns into 
\begin{align*}
&\norm[\Lom{∞}]{\ue(\cdot,t)-S_2(t)φ_b}\\
 &\le \f{K}{\sqrt{π}}\sqrt{\f{ε}t} + C\sqrt{εt} |φ_b| + \sqrt{ε}(\sqrt{ε}(1+t)+ \sqrt{t}+2t+2t\sqrt{t})\sqrt{π} \norm[X_t]{\ve} + 2\sqrt{εt} \norm[X_t]{\ve}
\end{align*}
for every $ε\in(0,ε_0)$ and $t>0$. Again, taking the supremum over $t\in(τ_1,τ_2)$ and using the boundedness of $\norm[X_{τ_2}]{\ve}$ according to \eqref{thm:ball-estimate}, we can conclude
Theorem~\ref{thm:upperbound-ball}.
\end{proof}

\section{A remark on condition \texorpdfstring{\eqref{decayofphi}}{(10)}. Long-time behaviour of the heat equation on the exterior of the ball.} \label{sec:remark-condition}
\begin{remark}
The main effect of the difference in geometry between $\Om=\ballexterior$ and $\Om=ℝ^N_+$ seems to lie in the additional condition $\sup_{ρ\ge 1} |ρφ(ρ)|<\infty$ to be posed on the initial data for the convergence result in the case of $\Om=\ballexterior$. 
In order to understand why this is not too unnatural, let us consider the solution of the heat equation emanating from initial data $φ=\charact{ℝ^3\setminus B_b(0)}$. 
\begin{align*}
 S_1(t)& \charact{ℝ^3\setminus B_b(0)} (r)
 = \f{1}{r\sqrt{4πt}} \kl{\int_b^{∞} ρ e^{-\f{(r-ρ)^2}{4t}} \drho  - \int_b^{∞} ρe^{-\f{(r+ρ-2)^2}{4t}} \drho }\\
 &= \f1{r\sqrt{4πt}} \kl{\int_{\f{b-r}{2\sqrt t}}^{∞} (2\sqrt{t}z+r) e^{-z^2} \dz - \int_{\f{b+r-2}{2\sqrt t}}^{∞} (2\sqrt{t} z +2-r ) e^{-z^2} 2\sqrt{t} \dz}\\
 &= \f1{r\sqrt{π}} \kl{2\sqrt{t} \int_{\f{b-r}{2\sqrt t}}^{\f{b+r-2}{2\sqrt t}} ze^{-z^2} \dz + r \int_{\f{b-r}{2\sqrt t}} ^{∞} e^{-z^2} \dz + (r-2) \int_{\f{b+r-2}{2\sqrt t}} ^{∞} e^{-z^2} \dz }\\
 &=\f{\sqrt t}{r\sqrt{π}} \kl{e^{-\f{(b-r)^2}{4t}}-e^{-\f{(b+r-2)^2}{4t}}} + \f1{\sqrt{π}} \int_{\f{b-r}{2\sqrt t}}^{∞} e^{-z^2}\dz + \f{(r-2)}{r\sqrt{π}} \int_{\f{b+r-2}{2\sqrt t}}^{∞} e^{-z^2} \dz.
\end{align*}
This means that 
\[
 S_1(t)φ→ 0 + \f12 + \f{r-2}{2r} = 1-\f1r \qquad \text{as } t\to ∞.
\]
In particular, $S_1\kl{\nf{t}{ε}}φ\not\to 0$ as $\nf{t}{ε}\to \infty$.
\end{remark}

\section{Lower estimates in the halfspace: Proof of Theorem \texorpdfstring{\ref{thm:lowerbound-halfspace}}{6}.}\label{sec:lowerbound-halfspace}
\begin{proof}[Proof of Theorem \ref{thm:lowerbound-halfspace}]
In order to show optimality of the rate $\sqrt{ε}$, we strive to find a lower estimate with the same rate, for one concrete example of initial data: 
We set $φ_b:=0$ and $φ:=\charact{\set{x_N>b}}$ for some $b>0$. Then, apparently, $Φ=φ-S_2(0)φ_b=φ$.

Due to $\ve$ being a fixed point of $\Qe$ (cf. Definition \ref{def:solution} and \eqref{defQe} or the construction in the proof of Theorem~\ref{thm:ball} and Remark~\ref{remark:existence-halfspace}), we have 
\[
 \ve(x,t)=\Qe[\ve](x,t) = S_1\kl{\nf{t}{ε}}φ(x) - 0 -\Det {\ve}(x,t) \ge S_1\kl{\nf{t}{ε}}φ(x) - |\Det {\ve}(x,t)|. 
\]
Letting $y_*:=(y',-y_N)$ for $y=(y',y_N)\in ℝ^{N-1}\times ℝ_+$, we can write the first of these terms explicitly, 
cf. \cite[(1.3) and (1.4)]{fila_ishige_kawakami_largediffusionlimit}: 

\begin{align*}
 & S_1\kl{\nf t{ε}} Φ (x)
 = \int_{ℝ^N_+} \kl{\nf{4πt}{ε}}^{-\f N2} \kl{e^{-\f{|x-y|^2}{\nf{4t}{ε}}} - e^{-\f{|x-y*|^2}{\nf{4t}{ε}}}} φ(y',y_N) \dy \\
   &= \int_{ℝ^N_+} \kl{\nf{4πt}{ε}}^{-\f N2} \kl{e^{-\f{|x'-y'|^2}{\nf{4t}{ε}}}e^{-\f{|x_N-y_N|^2}{\nf{4t}{ε}}} - e^{-\f{|x'-y'|^2}{\nf{4t}{ε}}} e^{-\f{|x_N+y_N|^2}{\nf{4t}{ε}}}} φ(y)\dy  \\
   &= \int_{ℝ^{N-1}} \kl{\nf{4πt}{ε}}^{-\f {N-1} 2} e^{-\f{|x'-y'|^2}{\nf{4t}{ε}}} \dy ' \int_b^{∞} (\nf{4πt}{ε})^{-\f12} \kl{e^{-\f{(x_N-y_N)^2}{\nf{4t}{ε}}}-e^{-\f{(x_N+y_N)^2}{\nf{4t}{ε}}}} \dy _N\\
  &= \int_{\sqrt{\f{ε}{4t}}(b-x_N)}^\infty \kl{\nf{4πt}{ε}}^{-\f12} e^{-z^2} \kl{\nf{4t}{ε}}^{\f12} \dz -\int_{\sqrt{\f{ε}{4t}(b+x_N)}}^{∞} π^{-\f12} e^{-z^2}\dz\\
  &= \f1{\sqrt{π}} \int_{\sqrt{\f{ε}{4t}}(b-x_N)}^{\sqrt{\f{ε}{4t}}(b+x_N)}e^{-z^2} \dz 
  =\f{x_N\sqrt{ε}}{\sqrt{πt}} I(ε,t,x_N),
\end{align*}
if we abbreviate 
\begin{equation*}
 I(ε,t,x_N):=\f{\sqrt{t}}{x_N\sqrt{ε}} \int_{\sqrt{\f{ε}{4t}}(b-x_N)}^{\sqrt{\f{ε}{4t}}(b+x_N)}e^{-z^2} \dz.
\end{equation*}



From \eqref{Dethalfspacesmallness} and Lemma \ref{lem:weminusS2} we obtain $C>0$ such that 
\begin{align*}
 \ue(x,t)&= \ve(x,t)+\we(x,t)\\
 & \ge 
 x_N\sqrt{\f{ε}{πt}} I(ε,t,x_N) 
 - Cε^{\f12}(t^{\f12}+ε^{\f12}t^{\f34}+t)\norm[X_t]{v} - 2\sqrt{εt}\norm[X_t]{\ve} 
\end{align*}
for all $(x,t)\in ℝ^N_+\times(0,\infty)$. According to \eqref{estimateXThalfspace}, for every $T>0$ we can hence find $\widetilde{C}>0$ such that 
\begin{align*}
 \ue(x,t)&\ge 
  x_N\sqrt{\f{ε}{πt}} I(ε,t,x_N) 
 - \widetilde{C}ε^{\f12}(t^{\f12}+ε^{\f12}t^{\f34}+t)\\
 &= \sqrt{\f{ε}{t}}\kl{ \f{x_N}{\sqrt{π}} I(ε,t,x_N) - \widetilde{C}(t)t^{\f12}(t^{\f12}+ε^{\f12}t^{\f34}+t)}
\end{align*}
holds for every $x\in ℝ^N_+$ and every $t\in(0,T)$.

After these preparations, we can begin the actual proof of the statement of
Theorem~\ref{thm:lowerbound-halfspace}: 
Given $x_0$ we pick $τ_2$ such that $\widetilde{C}(τ_2)τ_2^{\f12}(τ_2^{\f12}+τ_2^{\f34}+τ_2) < \f{x_0}{2\sqrt{π}}$ and for any compact $K\subset ℝ^{N-1}\times(x_0,∞)\times(0,τ_2)$ we let 
\[
C:=\f1{2\inf\set{t|\exists x \text{ with } (x,t)\in K}}\kl{\f{x_0}{2\sqrt{π}}-C(τ_2)τ_2^{\f12}(τ_2^{\f12}+τ_2^{\f34}+τ_2)}.
\] 
Noting that $I(ε,t,x_N)\to 1$ (uniformly with respect to $(x,t)\in K$) as $ε\to 0$, we choose $ε_0\in(0,1)$ so small that $I(ε,t,x_N)>\f12$ for all $(x,t)\in K$ and thus ensure that 
\[
  \ue(x,t) \ge C\sqrt{ε} \qquad \text{for every } (x,t)\in K \text{ and } ε\in(0,ε_0),
 \]
 as desired. 
\end{proof}

\section{Lower estimates in the exterior of the ball. Proof of Theorem~\texorpdfstring{\ref{thm:lowerbound-ball}}{7}.}\label{sec:lowerbound:exterior}

\begin{proof}[Proof of Theorem~\ref{thm:lowerbound-ball}]
We introduce $φ(r):=\f1r \charact{\set{r>b}}$ for some $b>1$ and $φ_b:=0$.
Obviously, $\sup_{r\ge 1} |rφ(r)|$ is finite, $Φ=φ$ and $\De{φ_b}=0$.

According to Definition \ref{def:solution} and \eqref{defQe}, $\ve$ is a fixed point of $\Qe$ and thus 
\[
 \ve(x,t)=\Qe[\ve](x,t) = S_1\kl{\nf{t}{ε}}φ(x) - 0 -\Det {\ve}(x,t) \ge S_1\kl{\nf{t}{ε}}φ(x) - |\Det {\ve}(x,t)|. 
\]

Again we compute the first term explicitly: 
\begin{align*}
 S_1\kl{\nf t{ε}} φ &= \f1{r\sqrt{4π\nf{t}{ε}}} \kl{\int_b^{∞} e^{-\f{(r-ρ)^2}{\nf{4t}{ε}}}\f{ρ}{ρ} \drho  - \int_b^{∞} e^{-\f{(r+ρ-2)^2}{\nf{4t}{ε}}}\drho }\\
 &= \f1{r\sqrt{π}} \int_{\f{b-r}{2\sqrt {\nf{t}{ε}}}}^{\f{b+r-2}{2\sqrt {\nf{t}{ε}}}} e^{-z^2} \dz 
= \f{r-1}{r\sqrt{π}} \sqrt{\f{ε}t} I(ε,t,r), 
\end{align*}
where 
\[
 I(ε,t,r) := \f{\sqrt t}{(r-1)\sqrt{ε}} \int_{\f{b-r}{2\sqrt {\nf{t}{ε}}}}^{\f{b+r-2}{2\sqrt {\nf{t}{ε}}}} e^{-z^2} \dz.
\]

%

By \eqref{eq:Detball} and Lemma \ref{lem:weminusS2}
\begin{align*}
\ue(r,t)= \ve(r,t)+\we(r,t) & \ge 
 \f{r-1}{r\sqrt{π}} \sqrt{\f{ε}t} I(ε,t,r) 
 - ε\sqrt{π}(1+t)\norm[X_T]{\ve}-2\sqrt{εt}\norm[X_t]{\ve} 
 \end{align*}
 for all $r\ge 1$, $t>0$; and according to \eqref{thm:ball-estimate} for every $T>0$ we can find $\widetilde{C}(T)$ such that 
 \begin{align*}
 \ue(r,t) &\ge 
 \f{r-1}{r\sqrt{π}} \sqrt{\f{ε}t} I(ε,t,r)
 - \widetilde{C} \sqrt{ε} (\sqrt{ε} (1+t) +\sqrt{t} )\\
\end{align*}
holds for every $r\ge 1$, $t\in(0,T)$ and $ε\in(0,ε_0)$. 
%
With these, we can easily prove the statement of Theorem~\ref{thm:lowerbound-ball}: 
Given $r>1$ we let $τ_2$ be such that $\f{r-1}{r} \sqrt{\f1{π}}\f12 - \widetilde{C}(τ_2) \sqrt{τ_2}(1+τ_2+\sqrt{τ_2}) =:c_1>0$. For any compact $K\subset (ℝ^3\setminus B_r(0) )\times(0,τ_2)$, we then use the (uniform in $K$) convergence 
$I(ε,t,r)\to 1$ 
as $ε\to 0$ to pick $ε_0\in(0,\f1{\sqrt{π}})$ such that $I(ε,t,x)>\f12$ for all $(x,t)\in K$. Then for $ε\in(0,ε_0)$ and $(x,t)\in K$, 
\begin{align*}
& \ue(x,t) \ge \f{r-1}{r} \sqrt{\f{ε}{πt}} \kl{\f{\sqrt{t}}{(r-1)\sqrt{ε}}\int_{\f{b-r}{2\sqrt {\nf{t}{ε}}}}^{\f{b+r-2}{2\sqrt {\nf{t}{ε}}}} e^{-z^2} \dz} -\widetilde{C} \sqrt{ε}(\sqrt{ε} (1+t) +\sqrt{t})\\
 & \ge \left(\f{r-1}{2r} \f1{\sqrt{π}} -\widetilde{C} \sqrt{t} \sqrt{ε}(\sqrt{ε} (1+t) +\sqrt{t})\right)\sqrt{\f{ε}{t}} 
 \ge \f{c_1}{\sqrt{t}}\sqrt{ε} \ge \f{c_1}{\sqrt{τ_1}}\sqrt{ε}
\end{align*}
if we define $τ_1:=\inf\set{t~|~\exists x: (x,t)\in K}$. 
Finally setting $C:=\f{c_1}{\sqrt{τ_1}}$, we obtain Theorem~\ref{thm:lowerbound-ball}.
\end{proof}

\noindent
Acknowledgements. The first author was supported in part by the Slovak
Research and Development Agency under the contract No. APVV-14-0378 and by the VEGA grant
1/0347/18. The second author was supported in part by the Grant-in-Aid for 
Scientific Research (A)(No.~15H02058)
from Japan Society for the Promotion of Science. The third author was supported by the Grant-in-Aid for Young Scientists (B)
(No.~16K17629)
from Japan Society for the Promotion of Science.


\end{document}